\documentclass[12pt]{amsart}
\usepackage[dvipdfmx]{graphicx, graphics}
\usepackage{amssymb, amscd,latexsym, color}
\usepackage{amsmath,amsthm,amssymb, ascmac, daytime, braket}
\newtheorem{thm}{Theorem}[section]
\newtheorem{lem}[thm]{Lemma}

\newtheorem{cor}[thm]{Corollary}
\newtheorem{ex}[thm]{Example}
\newtheorem{rmk}[thm]{Remark}

\begin{document}
\title []{Ribbon disks with the same exterior}
\author{Tetsuya Abe and Motoo Tange}
\address{Ritsumeikan University,  
1-1-1 Noji-higashi, Kusatsu Shiga 525-8577,  Japan}
\email{tabe@fc.ritsumei.ac.jp}
\address{Institute of Mathematics,  University of Tsukuba, 1-1-1 Tennodai, Tsukuba, Ibaraki 305-8571, Japan}
\email{tange@math.tsukuba.ac.jp}
\maketitle

 \begin{abstract}
We  construct
an infinite family of  inequivalent slice disks with the same exterior,
which gives an affirmative answer to an old question asked by
Hitt and Sumners in 1981.
Furthermore, we prove that  these slice disks are  ribbon disks. 
\end{abstract}
\section{Introduction}
One of the most outstanding  problems in knot theory has been whether 
 knots are determined by their complements or not.
The celebrated theorem of Gordon and Luecke \cite{GL}
states that 
if the complements of two classical  knots in the $3$-sphere  $S^3$ are diffeomorphic,
then  these knots are equivalent.
(For more  recent results, see  \cite{Gainullin, Kegel, Ravelomanana}.)
For higher-dimensional knots,
there exist at most two inequivalent  $n$-knots  
 ($n \ge 2$) with diffeomorphic exteriors, see \cite{Browder, Gluck, LS}.
Examples of such   $n$-knots are given in \cite{CS} and \cite{Gordon}.
\vskip 2mm

We consider an analogous problem for slice disks,
that is,  smoothly and properly embedded  disks  in the standard 4-ball $B^4$.
The situation is quite different.
Let $X$ be the exterior of a slice disk, 
and define $\zeta(X)$ to be the number of inequivalent slice disks 
whose  exteriors are diffeomorphic to $X$,
where two slice disks  $D_1$ and $D_2$ are  equivalent
if there exists a diffeomorphism  $g\colon B^4 \to B^4$
such that $g(D_1)=D_2$. 
In 1981, 
Hitt and Sumners  \cite[Section~4]{HS1}  asked the following.

\vskip 2mm
\noindent
\textbf{Question 1.}
Is there a slice disk exterior $X$ with $\zeta(X)=+\infty$?
\vskip 2mm

It seems that  no essential progress has been  made to Question $1$  until recently.
One of the reasons is  that when we consider Question 1, 
we often encounter the smooth Poincar\'{e} conjecture in dimension four,
which is one of the most challenging unsolved problems.
In 2015,   Larson and  Meier \cite{LM} produced 
 infinite families of inequivalent homotopy-ribbon disks with homotopy equivalent exteriors,
which gives a partial answer to Question~1.
In this paper, we prove the following.

\begin{thm}\label{thm:ribbon_disks}
There is a sequence of slice disks $D_n  (n \ge 0)$ satisfying the following properties:
\begin{enumerate}
\item  The exterior of each slice disk $D_n$ is diffeomorphic to that of $D_0$.
\item The knots $\partial D_n$ are  mutually inequivalent, therefore 
$D_n  $ are  mutually inequivalent.
\item$D_n$ is a ribbon disk.
\item  The knots $\partial D_n$ are obtained from $4_1 \# 4_1$ by the $n$-fold
annulus twist.
\end{enumerate}

\end{thm}

As an immediate corollary of (1) and (2) in Theorem \ref{thm:ribbon_disks}, we obtain the following.

\begin{cor}\label{cor:ribbon_disks}
There is a  slice disk exterior $X$ with $\zeta(X)=+\infty$.

\end{cor}

We do not know whether 
there exist 
 inequivalent  ribbon disks  with the same exterior
if  their boundaries  are equivalent.
However, there is a related result.
In 1991, Akbulut \cite{Akbulut} found an interesting pair of  ribbon disks using the
Mazur cork (see also  \cite[Subsection 10.2]{Akbulut2}).
Actually, 
he constructed a  pair of  ribbon disks  $E_1, E_2$ satisfying the following.
\vskip 2mm

\noindent
(1) The two knots $\partial E_1$  and  $\partial E_2$ coincide, that is, $\partial E_1=\partial E_2$.\\
(2) The two ribbon disks $E_1, E_2$ are NOT isotopic rel the boundary.\\
(3)  The two ribbon disks $E_1, E_2$ are equivalent, therefore, their exteriors are diffeomorphic.
 \vskip 2mm

\noindent
We will give  explicit pictures of the ribbon disks  which clarify the  symmetry
which comes from the Mazur cork
in the Appendix.\vskip 2mm

Here we give a historical remark on Question 1. 
In  \cite[Section 4]{HS1}, 
Hitt and Sumners   also asked  whether 
there exist  
 infinitely many higher-dimensional slice disks  with the same exterior
or not.
After pioneering work 
\cite{HS1, HS2, Plotnick},
 Suciu \cite{Suciu} proved that 
there exist infinitely many inequivalent $n$-ribbon disks  with the same exterior
for $n \ge 3$ in 1985, and Question $1$  remained open. \vskip 2mm

This paper is organized  as follows:
In Section 2, we recall some basic definitions.
In Section~3, we prove the first half of Theorem \ref{thm:ribbon_disks}.
In Section 4,   we prove  the latter  half of Theorem \ref{thm:ribbon_disks}.
In the Appendix, we will give  explicit pictures of Akbulut's ribbon disks  which clarify  the  symmetry.
\subsection*{Acknowledgments}
The first author was supported by JSPS KAKENHI Grant Number 25005998
and 16K17597.
This paper was partially written  during
his stay at KIAS in March 2016. 
He  deeply thanks  Min Hoon Kim and  Kyungbae Park
for their hospitality at KIAS.
The second  author was  supported by  JSPS KAKENHI Grant Number 24840006 and 26800031.
We thank  Kouichi Yasui for encouraging us to write this paper, 
Hitoshi Yamanaka for  careful reading of a draft of this paper and valuable comments,
Selman Akbulut for giving us a useful comment on his example of ribbon disks, and 
the referees for valuable comments.
\section{Basic definitions}\label{Section_def}

In this short section, we recall some basic definitions and background.
\vskip 2mm

 We define a \textit{slice disk} to be a smoothly and properly embedded disk $D  \subset B^4$, and  the boundary of $D$,  $\partial D \subset S^3$,   is called a \textit{slice knot}.
A knot $R \subset S^3$ is called \textit{ribbon} if it bounds an immersed disk $\Delta \subset S^3$ with only ribbon singularities.
For the definition of a ribbon singularity, see the left picture of Figure \ref{ribbon_singularity}.
By pushing the interior of $\Delta$ into the interior of  $B^4$, we obtain a slice disk whose boundary is $R$, and 
the resulting slice disk is called a \textit{ribbon disk}. 
It is well known that this ribbon disk is uniquely determined by  $\Delta \subset S^3$  
up to isotopy.
The \textit{Slice-Ribbon Conjecture},  also known as the Ribbon-Slice Problem,
 states that 
every slice knot is  a ribbon knot, namely,
it always bounds a  ribbon disk.
Our work is partially motivated by 
this conjecture\footnote{Hass  \cite{Hass} proved that   a slice disk is a ribbon disk 
if and only if it is isotopic to a minimal disk in $B^4$,
 see also \cite[Appendix B]{HKL}, \cite[p450]{Allard}.
Therefore the Slice-Ribbon Conjecture   might be   solved (affirmatively or negatively)
using geometric analysis or geometric measure theory.}.
Here we give an example of a ribbon knot.

\begin{ex}

The knot in the middle picture of Figure  \ref{ribbon_singularity}  is a ribbon knot since 
it bounds an immersed disk $\Delta \subset S^3$ as in the 
right picture of Figure \ref{ribbon_singularity}.

\end{ex}

\begin{figure}[htp!]
\includegraphics[width=1\textwidth]{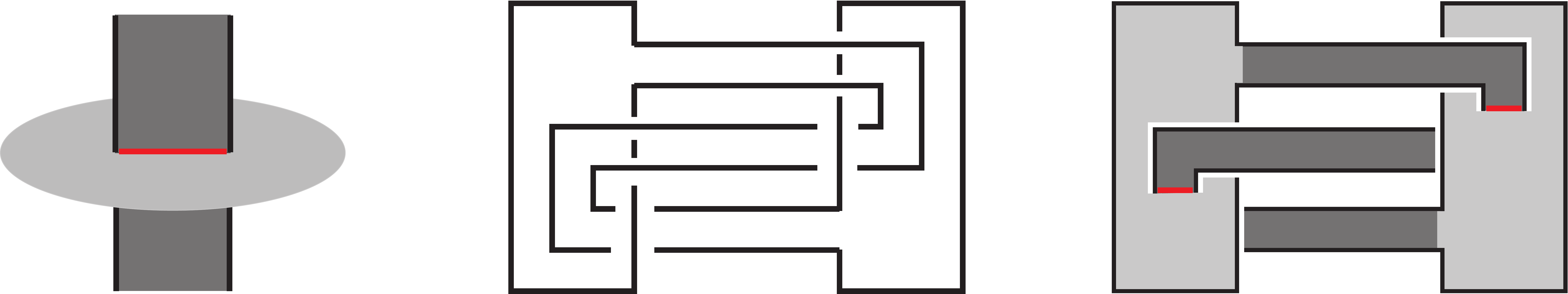}
\caption{A ribbon singularity (colored red), a ribbon knot, and  an immersed disk $\Delta \subset S^3$.}
\label{ribbon_singularity}
\end{figure}

Two knots $K_1, K_2$  are \textit{equivalent}
if there exists a diffeomorphism  $f\colon S^{3} \to S^{3}$
such that $f(K_1)=K_2$.
The \textit{exterior} of a knot $K$  is the $3$-manifold $S^{3} \setminus \nu(K)$,
where $\nu(K)$ is an open  tubular neighborhood of $K$ in $S^{3}$.
It is  unique up to diffeomorphisms,
and its interior is  diffeomorphic to 
the complement of $K$.
Similarly two slice disks  $D_1, D_2$ are \textit{equivalent}
if there exists a diffeomorphism  $g\colon B^4 \to B^4$
such that $g(D_1)=D_2$, and 
the \textit{exterior} of a slice disk  $D$  is the $4$-manifold 
$B^4 \setminus \nu(D)$, where $\nu(D)$ is an open  tubular neighborhood of $D$ in $B^4$.
Note that two (ambient) isotopic   slice disks  $D_1$  and $D_2$ are equivalent.

\section{Proof of the first half of Theorem \ref{thm:ribbon_disks}}

We prove the first half of Theorem \ref{thm:ribbon_disks},
that is, the statements (1) and (2) in Theorem \ref{thm:ribbon_disks}.
\vskip 3mm

Throughout  this paper,  we only consider  a 2-handlebody $HD$ which consists  of a $0$-handle, 1-handles,
and $2$-handles.
Also, note that the handle diagram of  $HD$ is drawn in the boundary of the $0$-handle.

\begin{proof}[Proof of the first half of Theorem \ref{thm:ribbon_disks}]
\begin{figure}[htp!]
\centering
\includegraphics[width=0.9 \textwidth]{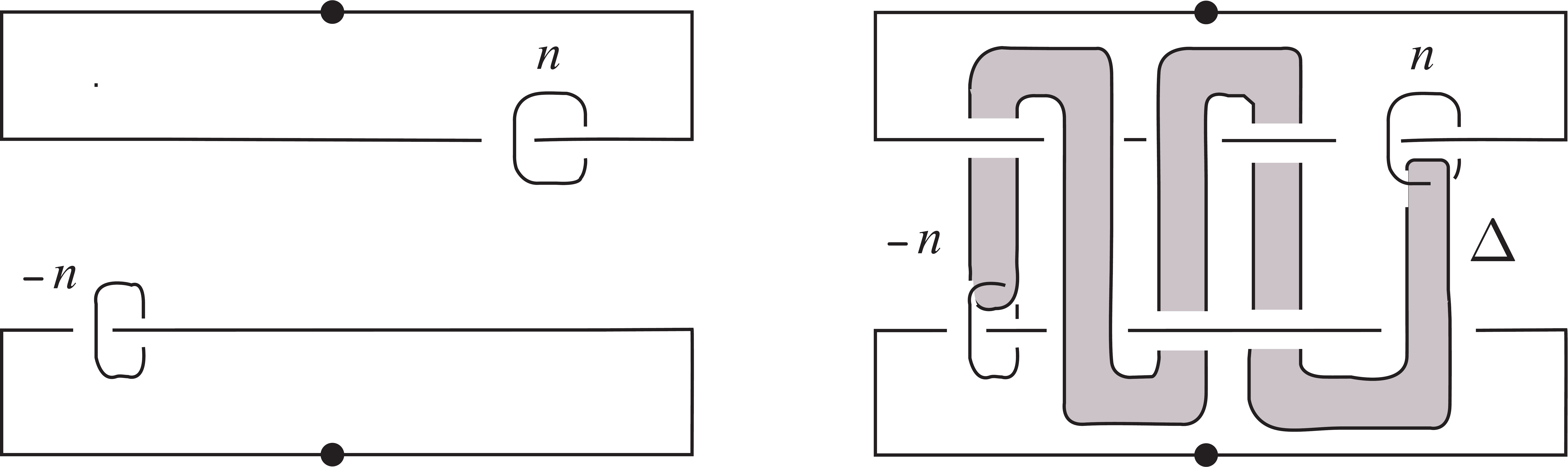}
\caption{A handle diagram of $B^4$ and the definition of the slice disk $D_n$.
}
 \label{ribbon_disks3}
\end{figure}
We consider the handle decomposition  of $B^4$ given by 
the handle diagram  in the  left picture in Figure  \ref{ribbon_disks3}.
Let $D_n$  be the slice disk obtained from the disk  $ \Delta$ in  the  right picture in Figure  \ref{ribbon_disks3}
by pushing the interior of $ \Delta$ into the interior of the $0$-handle,
and $X_n$ the exterior of $D_n$.

We will prove the statement (1) :  $X_n$  is diffeomorphic to $X_0$.
By the definition of dotted circles,  
$X_n$ is represented by the picture in 
Figure  \ref{fig:exterior} (after an isotopy).
For the dotted circle notation, 
see  \cite{Akbulut, GS}.
Then $X_n$    is diffeomorphic to $X_0$, which
 follows from the well-known fact that
\begin{figure}[ht!]
\centering
\includegraphics[width=0.4 \textwidth]{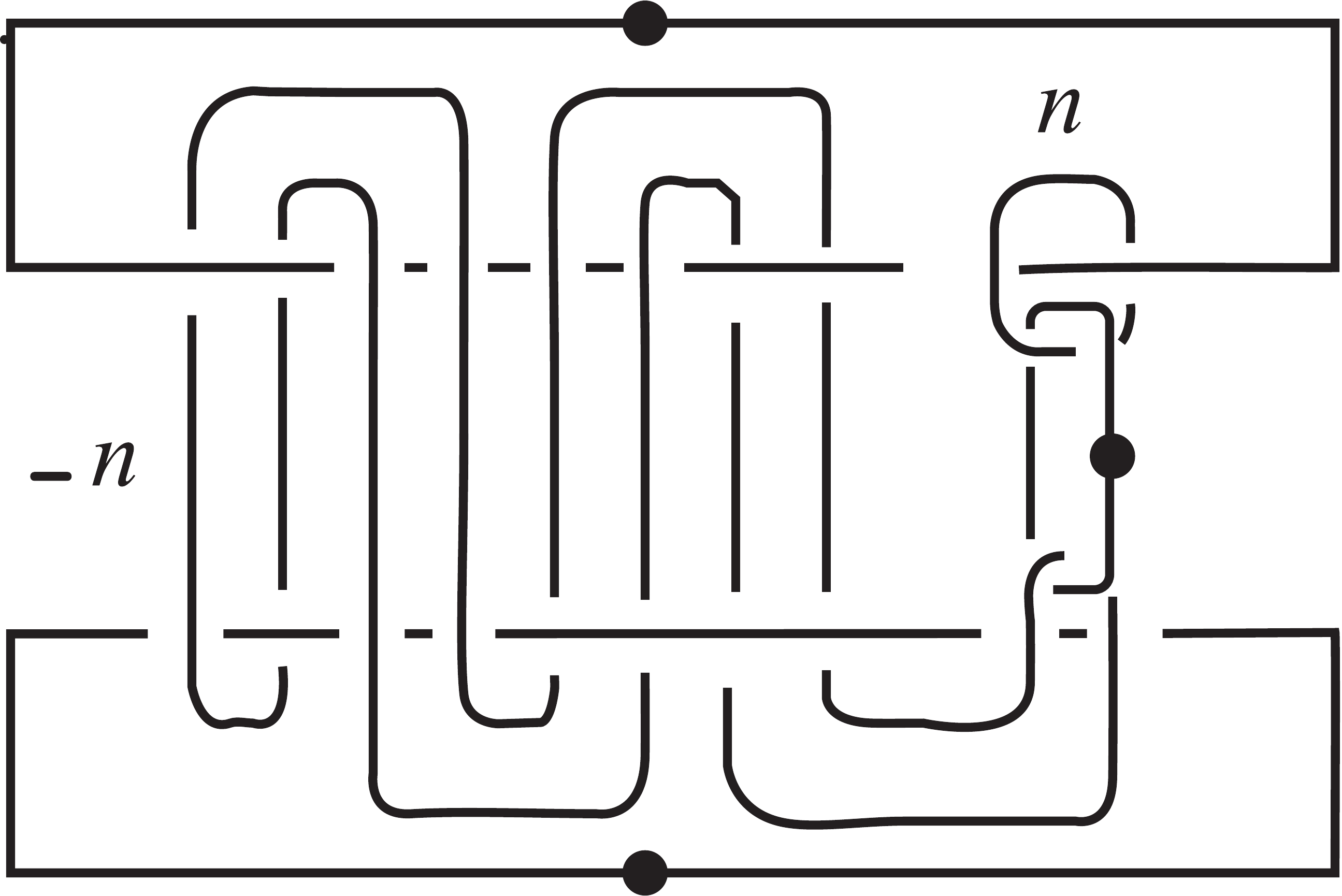}
\caption{A handle diagram of $X_n$.}
 \label{fig:exterior}
\end{figure}
 two handle diagrams that differ locally as in Figure  \ref{fig:Tips} are related by 
a sequence of handle moves,
see \cite{Akbulut, GS}. 
\begin{figure}[ht!]
\centering
\includegraphics[width=0.6 \textwidth]{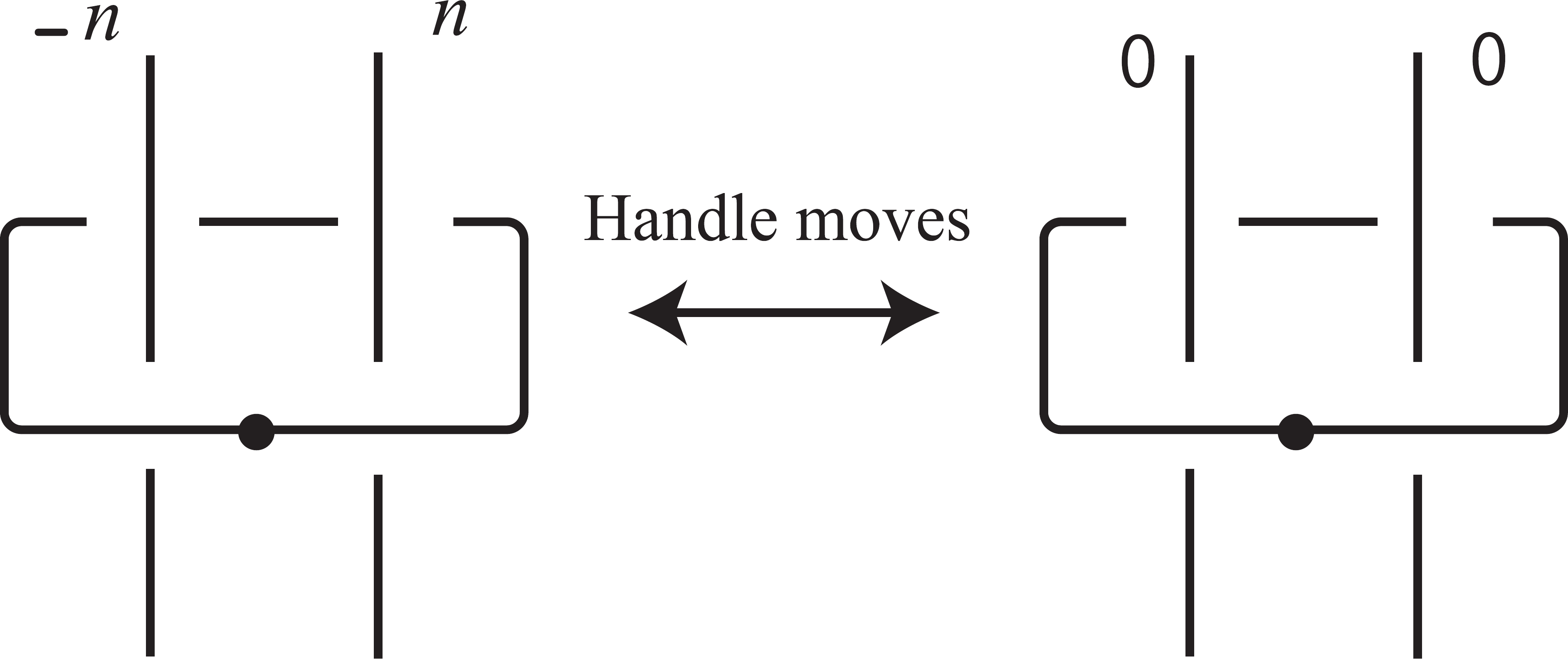}
\caption{Two handle diagrams related by a sequence of handle moves.}
 \label{fig:Tips}
\end{figure}

Recently, Takioka \cite{Takioka} calculated the $\Gamma$-polynomial\footnote{
This is a  polynomial invariant introduced by Akio Kawauchi in \cite{Kawauchi},
which is a specialization of the Homflypt polynomial.
This polynomial is independent of the Alexander polynomial  and the Jones polynomial,
and there is a polynomial time complexity algorithm for computing  the $\Gamma$-polynomial,
 see \cite{Przytycki}. }
 of the knots $\partial D_n$.
In particular,  he showed that the span of 
the $\Gamma$-polynomial of $\partial D_n$ is $2n+4$ and proved that 
the knots $\partial D_n$ are mutually inequivalent for $n \ge 0$,
which implies the statement (2). 
\end{proof}

\begin{rmk}
It is straightforward to see that 
the knot $\partial D_0$ is $4_1 \# 4_1$, that is,
the connected sum of two figure-eight knots.
We can also prove that $\partial D_n$ is isotopic to $\partial D_{-n}$ (by using  the symmetry of $4_1 \# 4_1$).
\end{rmk}

\begin{rmk}
It is not difficult to see that $X_n$ is diffeomorphic  to  the exterior of the   ribbon disk represented by
the  picture in Figure \ref{fig:exterior3}
(see  \cite[Subsection 1.4] {Akbulut} or \cite[Subsection 6.2]{GS}).
However, this fact does not a priori imply that $D_n$ is a ribbon disk.
In our case, $D_n$ is   a ribbon disk as proven later.

\begin{figure}[ht!]
\centering
\includegraphics[width=0.45 \textwidth]{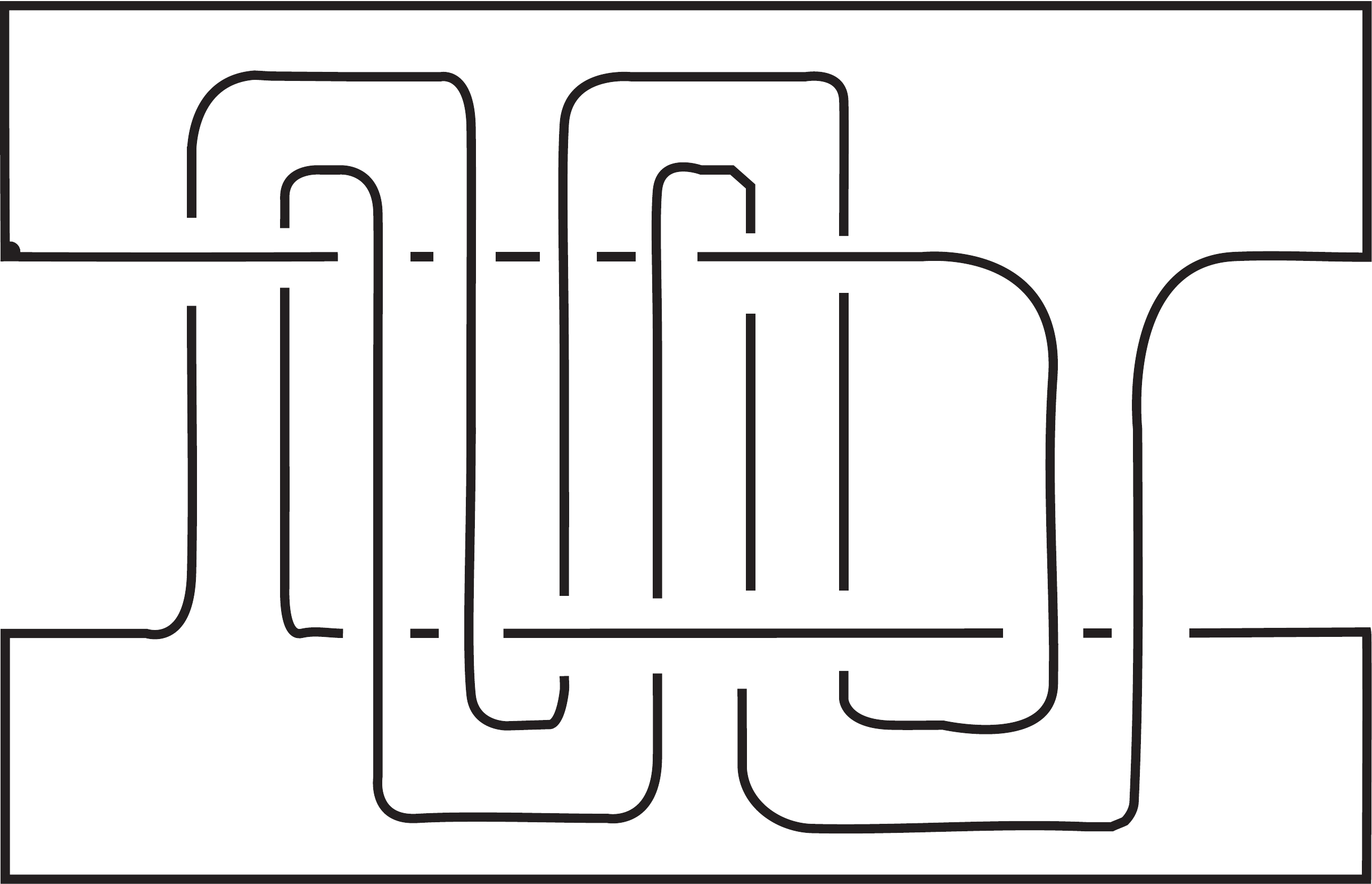}
\caption{A ribbon knot, which determines an obvious  ribbon disk.}
 \label{fig:exterior3}
\end{figure}

\end{rmk}
We conclude this section by asking the following question.
\vskip 2mm

\noindent
\textbf{Question 2.}
Let $D$ be a slice disk whose exterior is diffeomorphic to 
a ribbon disk  exterior.
Then is $D$ a ribbon disk ?

\section{Proof of the latter  half of Theorem \ref{thm:ribbon_disks}}\label{Section_ribbon}

In this section,  
we observe  Lemmas \ref{lem:operation1} and  \ref{lem:operation2}
which  are important when we deal with ribbon disks in terms of handle diagrams of $B^4$,
and prove  the latter  half of Theorem \ref{thm:ribbon_disks}.\\

\begin{lem}\label{lem:operation1}
Let $F$ and $F'$  be the two  smooth  disks in
the handle diagram of (a handle decomposition of)  $B^4$   which locally differ  as shown in Figure~\ref{fig:operation1},
 and $D$ and $D'$ be  the slice disks  obtained by pushing the interiors of $F$ and $F'$ 
 into the interior of the 0-handle.
 Then  $D$ and $D'$ are   (ambient)   isotopic in $B^4$.
\begin{figure}[htp!]
\includegraphics[width=0.85\textwidth]{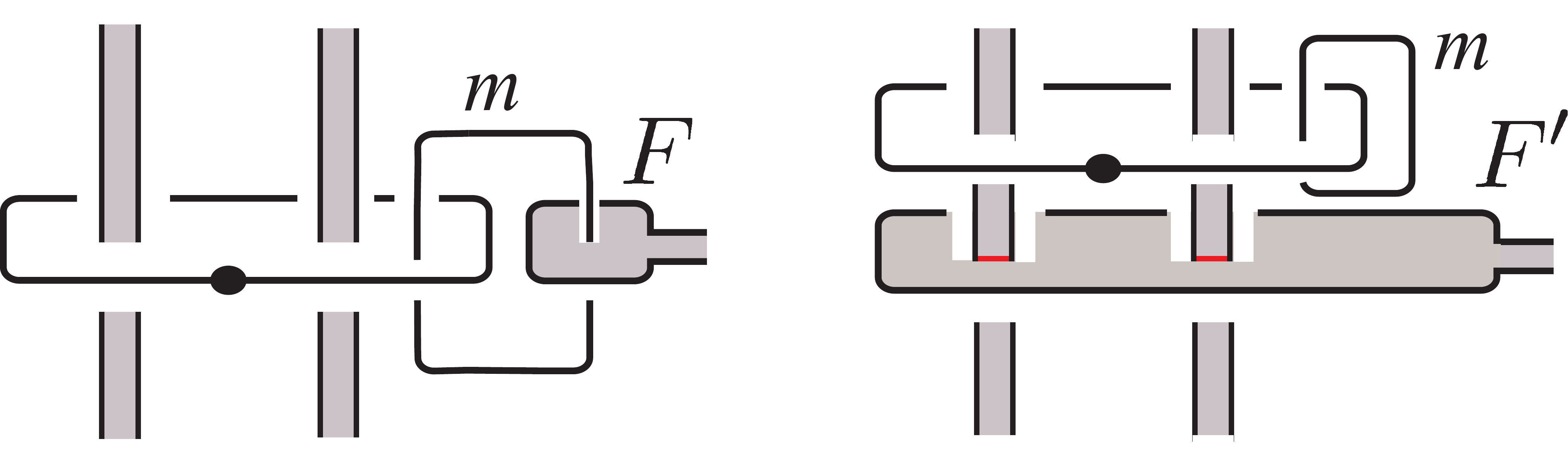}
\caption{Two smooth disks  $F$ and $F'$ in
the handle diagram of $B^4$.}
\label{fig:operation1}
\end{figure}
\end{lem}

\begin{proof}
See Figure \ref{fig:PROOF_ribbon_sing_in_HD}. 
\end{proof}
\begin{figure}[htp!]
\includegraphics[width=0.89\textwidth]{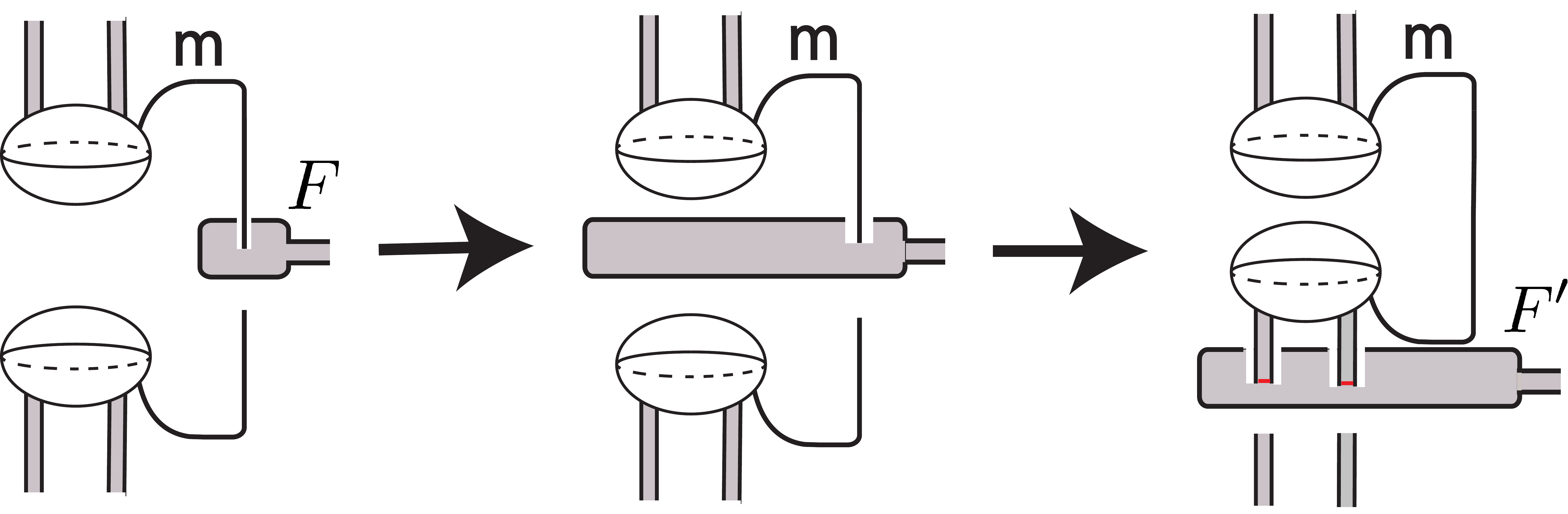}
\caption{An isotopy between  $D$ and $D'$ which are projected to
$F$ and $F'$, respectively.}
\label{fig:PROOF_ribbon_sing_in_HD}
\end{figure}

\begin{lem}\label{lem:operation2}
Let $F$ and $F'$  be two  smooth  disks
in handle diagrams  of $B^4$ which locally differ  as shown in Figure~\ref{fig:operation3},
and $D$ and $D'$ be  the slice disks  obtained by pushing the interiors of $F$ and $F'$ 
 into the interior of the 0-handle.
 Then  $D$ and $D'$ are (ambient) isotopic in $B^4$.
\begin{figure}[htp!]
\includegraphics[width=0.75\textwidth]{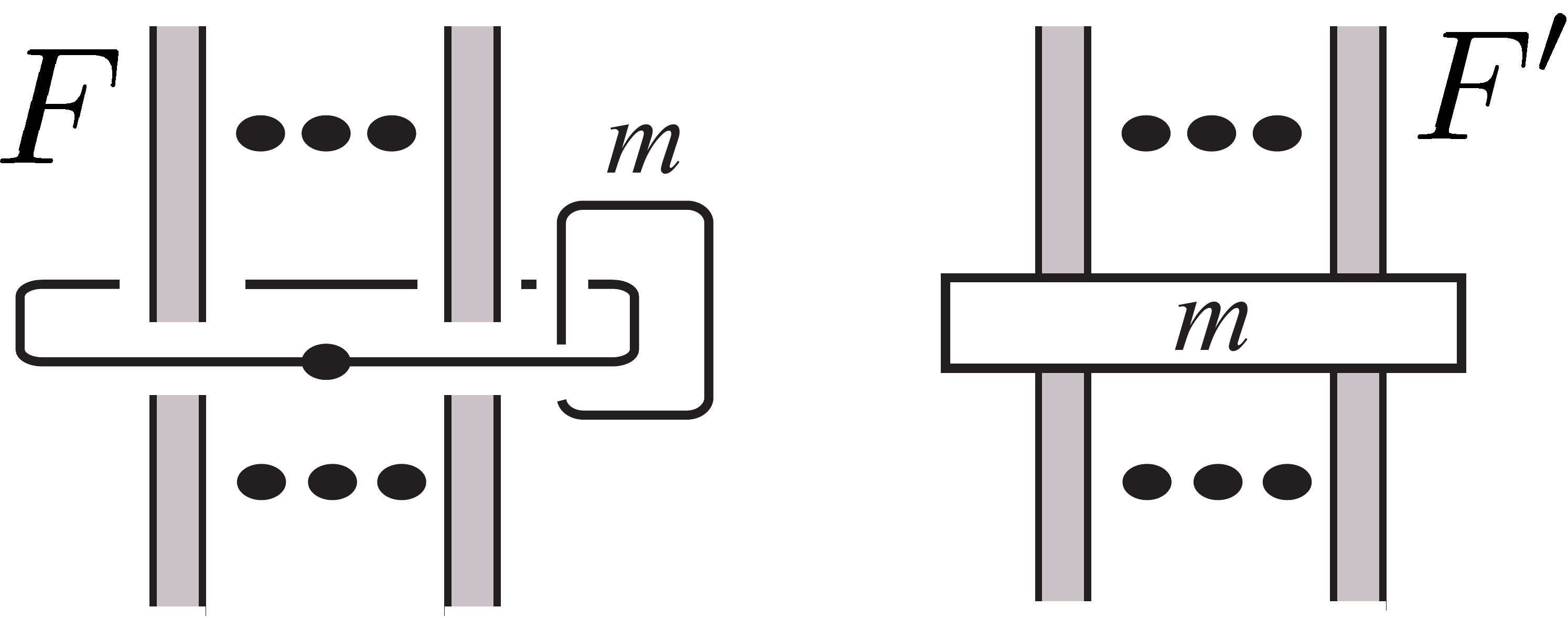}
\caption{Two disks $F$ and $F'$. Here the box named $m$
means the $m$-full twists. }
\label{fig:operation3}
\end{figure}
\end{lem}

\begin{proof}
After sliding  $D$ over the 2-handle in the left picture of Figure \ref{fig:operation3}, 
we cancel the 1/2-handle pair.  
Then we obtain the slice disk $D'$ which are projected to $F'$, see the right picture of Figure \ref{fig:operation3}.
\end{proof}
We are ready to prove   the latter half of Theorem \ref{thm:ribbon_disks}.

\begin{proof}[Proof of the latter half of Theorem \ref{thm:ribbon_disks}]

Recall that  the slice disk $D_n$ is obtained  by 
pushing the interior of  the disk $\Delta$ in the left picture of
Figure  \ref{ribbon_disks4} into the interior of  the 0-handle. 
By Lemmas \ref{lem:operation1} and \ref{lem:operation2},
$D_n$ is isotopic to the slice disk $D_n'$ which is projected 
to  the smooth disk  $\Delta'$ in the right picture of
Figure  \ref{ribbon_disks4}.
This implies that $D_n$ is a ribbon disk, and we have the statement (3).
Here we note that the smooth disk $\Delta'$ has four ribbon singularities, however,
we do not draw the whole picture of $\Delta'$ for simplicity.

\begin{figure}[htp!]
\vskip -90mm
\centering
\includegraphics[width=0.95 \textwidth]{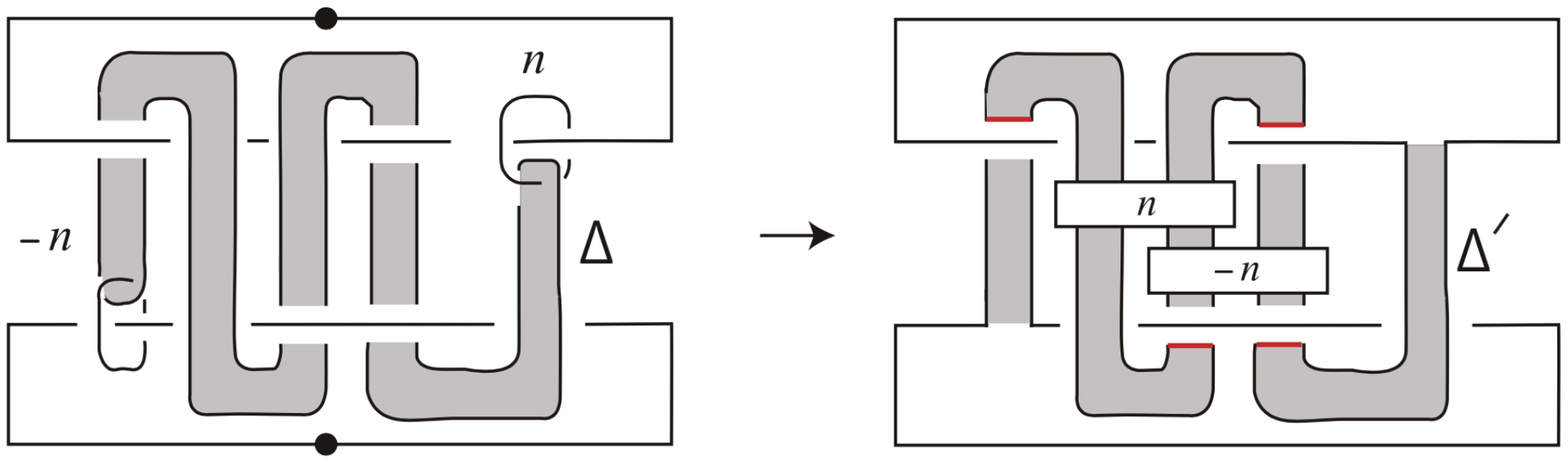}
\vskip -92mm

\caption{$\Delta'$  is an immersed disk with four ribbon singularities.
}

 \label{ribbon_disks4}
\end{figure}

The right picture of Figure  \ref{ribbon_disks4} tells us that 
the knot $ \partial D_n$ is isotopic to that in Figure  \ref{ribbon_disks6}.
\begin{figure}[htp!]
\centering
\includegraphics[width=0.5 \textwidth]{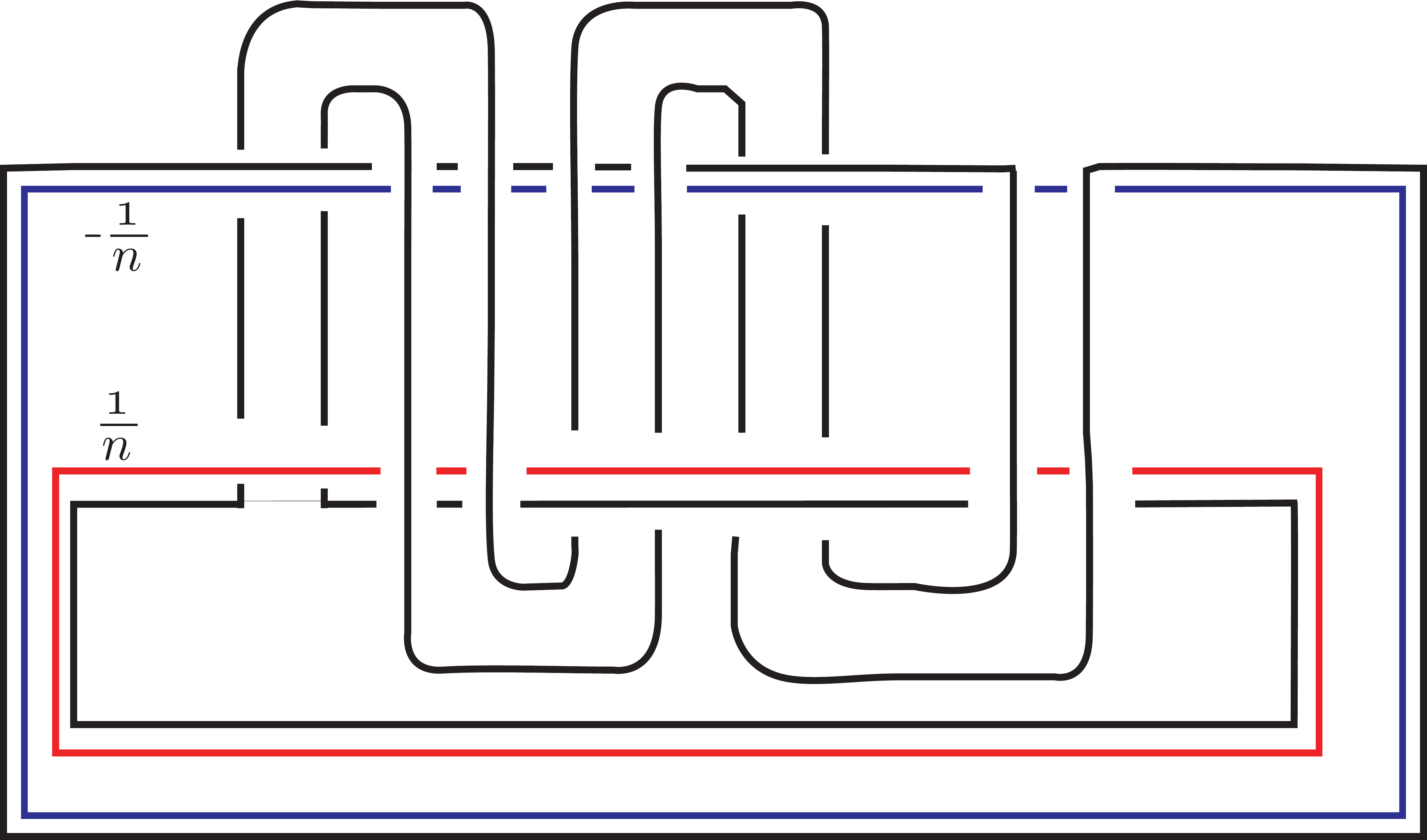}
\caption{}
 \label{ribbon_disks6}
\end{figure}
This implies that $ \partial D_n$ is obtained from the ribbon knot $ \partial D_0=4_1  \# 4_1$
by $n$-fold annulus twist, see \cite{AT}.
We have now obtained  the statement (4).
\end{proof}

\begin{rmk}\label{rem:annulus}
It turns out that 
the  knots $ \partial D_n$ are the same as  those in  \cite[Figure 7]{Osoinach}, 
however, we omit the proof since this fact is not used in this paper.
\end{rmk}


\section*{Appendix}
Akbulut \cite{Akbulut} constructed a pair of  ribbon disks $E_1$ and $E_2$
as in Figure  \ref{ribbon-disk},
where ribbon disks are  specified by dashed arcs.
\begin{figure}[htp!]  \centering 
\includegraphics[width=0.6 \textwidth]{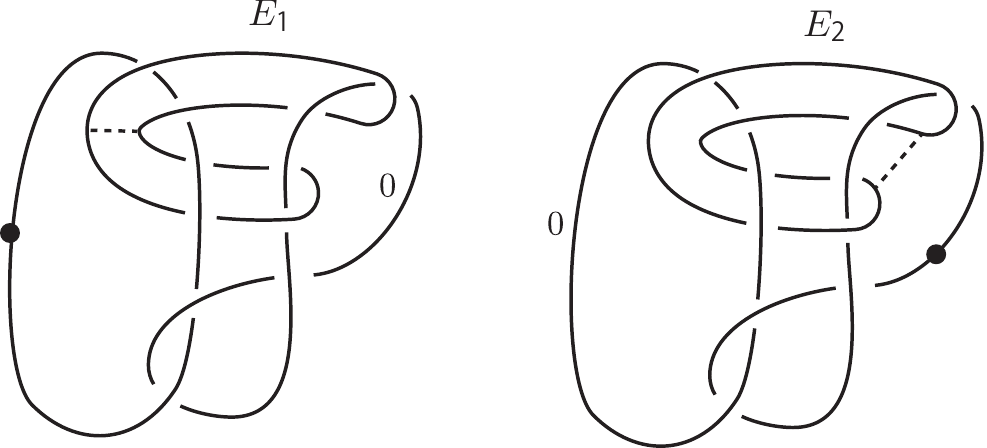}
\caption{A pair of ribbon disks $E_1$ and $E_2$.}  \label{ribbon-disk}
\end{figure}
In  \cite{Akbulut}, he  gave explicit pictures
of his ribbon disks; however, the symmetry which comes from the Mazur cork was not clear.
Here we give explicit pictures of $E_1$ and $E_2$ which clarify the  symmetry
as in Figure  \ref{anticork_ribbondisks}.
By the  symmetry, $E_1$ and $E_2$ are equivalent, hence  the ribbon disk exteriors are  diffeomorphic.
\begin{figure}[htp!] \centering
\includegraphics[width=0.75 \textwidth]{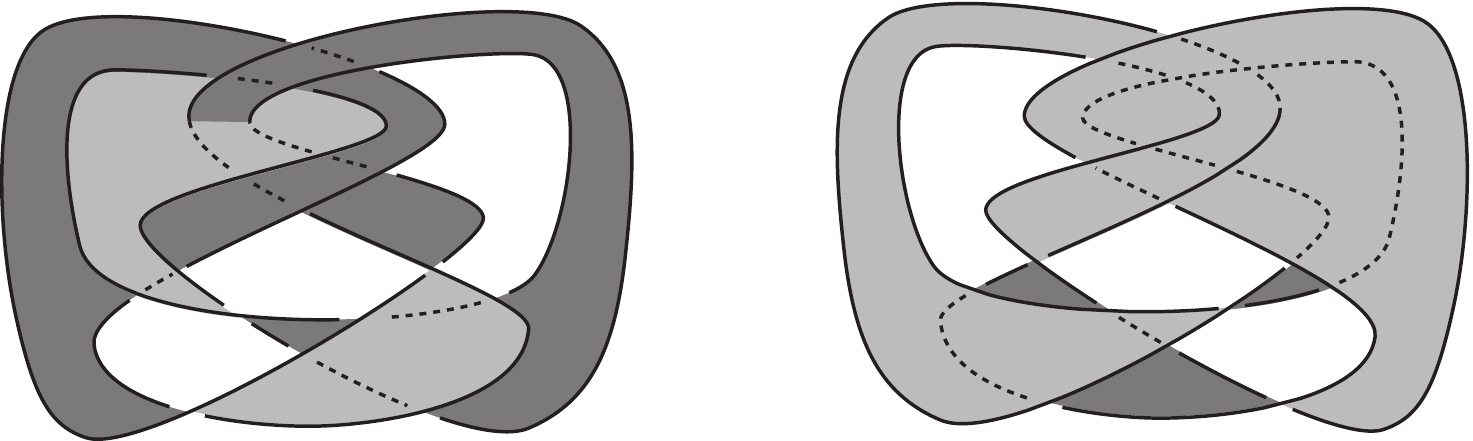} 
\caption{Explicit pictures of $E_1$ and $E_2$ which clarify the  symmetry.}  \label{anticork_ribbondisks}
\end{figure}
Figures \ref{isotopy} and \ref{isotopy2}  explain how to obtain 
explicit pictures of $E_1$ and $E_2$.

\begin{figure}[htp!]
\centering
\includegraphics[width=0.95 \textwidth]{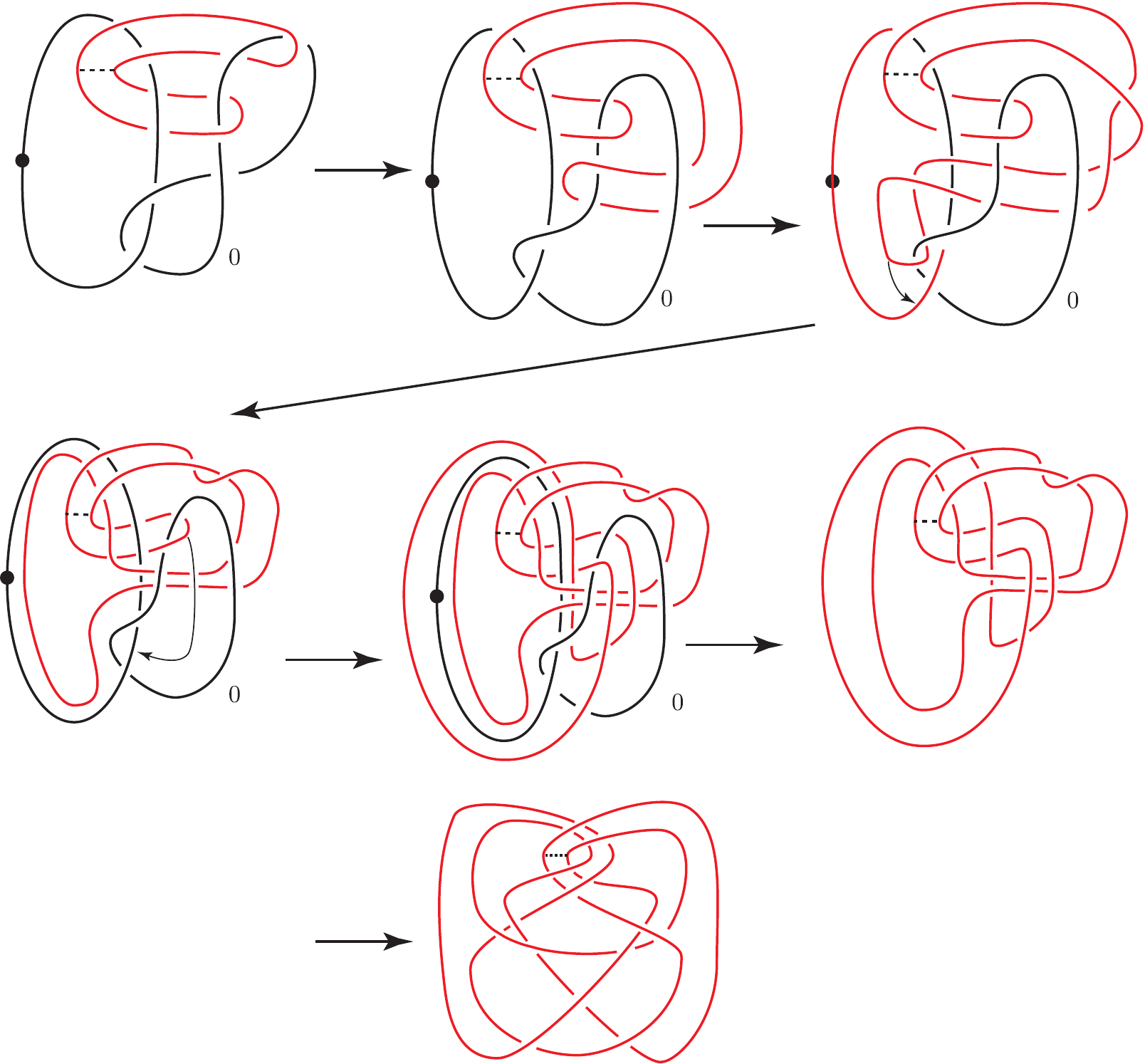}
\caption{An isotopy.}  \label{isotopy}
\end{figure}

\begin{figure}[htp!]
\vskip -70mm

\centering
\includegraphics[width=1 \textwidth]{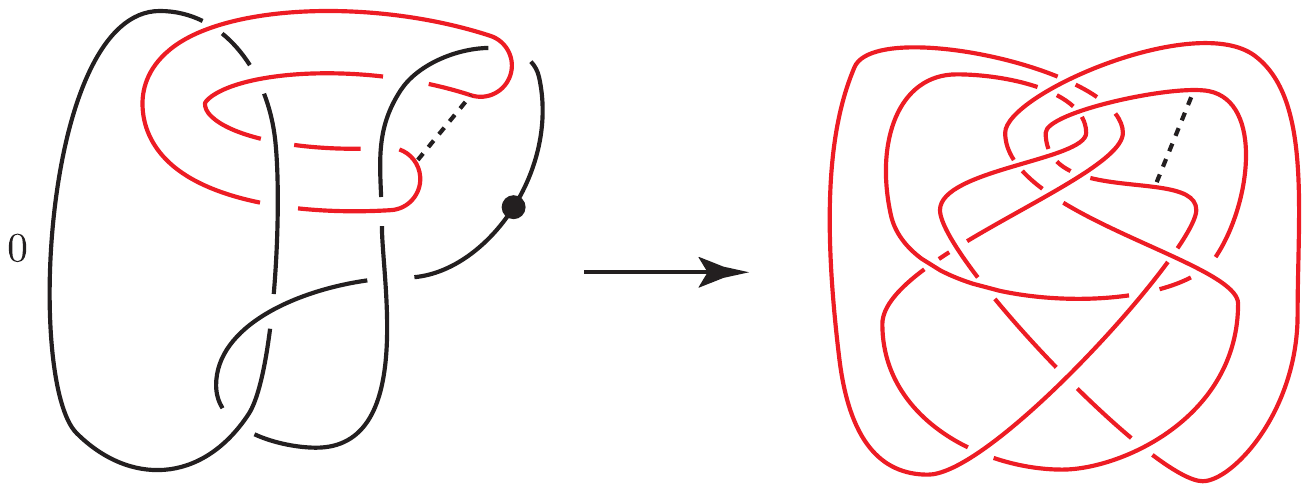}
\vskip -115mm
\caption{An isotopy.}  \label{isotopy2}

\end{figure}


%
%



\begin{thebibliography}{10}

\bibitem{AT} T. Abe and M. Tange,
\textit{A construction of slice knots via annulus twists},
 Michigan Math. J.
\textbf{65} (2016),  573--597.


\bibitem{Akbulut}
S.~Akbulut,
 \textit{A solution to a conjecture of Zeeman},
Topology, \textbf{30}, (1991), 513--515.



\bibitem{Akbulut2}
S.~Akbulut,
 \textit{$4$-manifolds},
Oxford University Press, Oxford, 2016.



\bibitem{Allard}
W.~K.~Allard, F.~J.~ Almgren, eds.,
\textit{Geometric measure theory and the calculus of variations}, 
Proc. Syrup. Pure Math. \textbf{44}, 
Amer. Math. Soc., Providence, RI, 1986.


\bibitem{Browder}
W.~Browder,
\textit{Diffeomorphisms of 1-connected manifolds},
Trans. Amer. Math. Soc. \textbf{128} (1967), 155--163

\bibitem{CS}
 S. Cappell and J. Shaneson, 
\textit{There exist inequivalent knots with the same complement}, 
Ann. Math. \textbf{103} (1976), 349--353.






\bibitem{DN} J.~Davis and S.~Naik,
\textit{Alexander polynomials of equivariant slice and ribbon knots in $S^3$}, 
Trans. Amer. Math. Soc. \textbf{358} (2006), 2949--2964


\bibitem{Gainullin}
F. Gainullin, 
\textit{Heegaard Floer homology and knots determined by their complements},
arXiv:1504.06180v4 [math.GT].




\bibitem{Gluck}
H.~ Gluck, 
\textit{The embedding of two-spheres in the four-sphere}, 
Trans. Amer. Math. Soc. \textbf{104} (1962), 308--333.







\bibitem{GS} R. Gompf and A.~Stipsicz, 
{\it 4-manifolds and Kirby calculus},
Graduate Studies in Mathematics, 20. American Mathematical Society, Providence, RI, 1999.


\bibitem{Gordon}
C. McA. Gordon,
\textit{Knots in the 4-sphere},
 Commentarii Math. Helvetici \textbf{51} (1976), 585--596.

\bibitem{GL}
C.~McA.~Gordon and J.~Luecke, 
\textit{Knots are determined by their complements},
 J. Amer. Math. Soc. \textbf{2} (1989),  371--415.

\bibitem{Hass}
J.~Hass,  
\textit{The geometry of the slice-ribbon problem},  
Math. Proc. Cambridge Philos. Soc. \textbf{94} (1983),  101--108.



\bibitem{HKL}
M.~Hedden, P.~Kirk, and C.~Livingston C,  
\textit{Non-slice linear combinations of algebraic knots}, J. Eur. Math. Soc.  \textbf{14} (2012) 1181--1208.

\bibitem{HS1}
L.~R.~Hitt and D.~W.~Sumners,
\textit{Many different disk knots with the same exterior},
Comment.~Math. Helvetici \textbf{56} (1981), 142--147.


\bibitem{HS2}
L. R. Hitt and D. W. Sumners,
\textit{There exist arbitrarily many different disk knots with the
same exterior},
Proc. Amer. Math. Soc. \textbf{86} (1982), 148--150.

\bibitem{HS}
J.~F.~P.~Hudson and D.~W.~L.~Sumners,
\textit{Knotted ball pairs in unknotted sphere pairs}, 
J.~London Math.~Soc.~\textbf{41} (1966), 717--722.

\bibitem{Kawauchi}
A.~Kawauchi,
\textit{On coefficient polynomials of the skein polynomial of an oriented link}, Kobe J. Math. 
\textbf{11} (1994), 49--68.

\bibitem{Kegel}
M.~Kegel,
\textit{The Legendrian knot complement problem},
arXiv:1604.05196v1 [math.GT].

\bibitem{LS}
R.~K.~Lashof and J.~L.~Shaneson, 
\textit{Classification of knots in codimension two}, 
Bull. Amer. Math. Soc. \textbf{75} (1969), 171--175.

\bibitem{LM}
K.~Larson and J.~Meier,
\textit{Fibered ribbon disks}, 
J. Knot Theory Ramifications \textbf{24}, 1550066 (2015) [22~pages].









\bibitem{Osoinach}
J. Osoinach, 
{\em Manifolds obtained by surgery on an infinite number of knots in $S^3$}, 
Topology {\bf 45} (2006), 725--733. 

\bibitem{Plotnick}
P. Plotnick,
\textit{Infinitely many disk knots with the same exterior}, 
Math. Proc. Cambridge Philos. Soc. \textbf{93} (1983),  67--72.


\bibitem{Przytycki}
J.~H.~Przytycki, 
\textit{The first coefficient of Homflypt and Kauffman polynomials}, 
arXiv:1707.07733v1 [math.GT].





\bibitem{Ravelomanana}
H.~C.~Ravelomanana,
\textit{Knot Complement Problem for L-space $\mathbb{Z}HS^3$},
 arXiv:1505.00239v5 [math.GT].


\bibitem{Suciu}
 A.~I.~Suciu,
\textit{Infinitely many ribbon knots with the same fundamental group},
Math. Proc. Cambridge Philos. Soc., \textbf{98}  (1985),  481--492.


\bibitem{Takioka}
H.~Takioka, \textit{Classification of Abe-Tange's ribbon knots},
preprint (2017).






\end{thebibliography}
\end{document}